\documentclass[12pt,oneside]{amsart}

\usepackage{amssymb,amsfonts, verbatim,mathrsfs,sansmath}
\usepackage[all,arc]{xy}
\usepackage{enumerate}
\usepackage{mathrsfs}

\usepackage{hyperref}
\usepackage[all]{xy}
\usepackage{amsthm}
\usepackage{comment} 
\usepackage{amssymb}
\usepackage[charter]{mathdesign}
\usepackage{tikz-cd}
\usepackage[all]{xy}

\usepackage{geometry}
\usepackage{amsmath}
 
\allowdisplaybreaks
\parskip \baselineskip
\usepackage{graphicx,color}

\usepackage{bm}
\newtheorem{thm}{Theorem}[section]

\theoremstyle{definition}
\newtheorem{definition}[thm]{Definition}

\theoremstyle{remark}
\newtheorem{remark}[thm]{Remark}

\renewcommand{\overline}[1]{\langle #1\rangle}

\makeatletter
\let\c@equation\c@thm
\makeatother
\numberwithin{equation}{section}

\title{Why is the category of near-vector spaces abelian?} 
\begin{document}
\begin{abstract} In this paper we present a unified proof of the fact that the category of modules over a ring and the category of near-vector spaces in the sense of J.~Andr\'e, over an appropriate scalar system (a `scalar group'), are both abelian categories. The unification is possible by viewing each of these categories as subcategories of the (abelian) category of modules over a multiplicative monoid $M$. Although in the case of near-vector spaces all elements of $M$ except one (the `zero' element) are invertible, we show that this requirement is not necessary for the corresponding category to be abelian in analogy to the well-known fact that modules over a ring form an abelian category even if the ring is not a field (i.e., modules over it are not vector spaces).

\end{abstract}

\author{Zurab Janelidze}
\address{Department of Mathematical Sciences, Stellenbosch University, South Africa}\address{National Institute for Theoretical and Computational Sciences (NITheCS), South Africa}
\email{zurab@sun.ac.za}

\author{Sophie Marques}
\address{Department of Mathematical Sciences, Stellenbosch University, South Africa}\address{National Institute for Theoretical and Computational Sciences (NITheCS), South Africa}
\email{smarques@sun.ac.za}

\author{Daniella Moore}
\address{Department of Mathematical Sciences, Stellenbosch University, South Africa}
\email{dmoore@sun.ac.za}

\makeatletter
\@namedef{subjclassname@2020}{
  \textup{2020} Mathematics Subject Classification}
\makeatother










\maketitle 

\newcommand{\sophie}[1]{{\color{red}\sf $\clubsuit\clubsuit\clubsuit$ Sophie: [#1]}}

\maketitle

\section{Introduction}\label{d1}

A `near-vector space' is a generalization of a vector space, introduced by J.~Andr\'e \cite{Andre}. These structures seem to mimic vector spaces in many ways, while at the same time encompassing interesting exotic examples (see \cite{Howell,Howell2, HM14, HS18, HM22, HR22, MM22, BMM25}). The starting point of our joint work was to explore whether near-vector spaces (over a fixed `scalar group') form an abelian category, just like the category of vector spaces (over a fixed field) do. Of course, behind the fact that the category of vector spaces is abelian is a more general fact that the category of modules over any ring is abelian. 
We were able to establish abelianness for the category of what we ended up calling in this paper \emph{Andr\'{e} modules} over a \emph{multi-near-ring}, new structures that are in a similar relation to modules over a ring, as near-vector spaces are to vector spaces over a field. 
In fact, the axioms for an Andr\'e module are chosen in such a way that abelianness of the corresponding category becomes almost an obvious fact. What is less obvious is the link between our axioms defining an Andr\'e module and Andr\'e's original axioms defining a near-vector space. This link relates to an open question whether the class of near-vector spaces is closed under subspaces, which was resolved in \cite{HM22, rabie, HR22, MM22}: it rests on Theorem~\ref{ThmA}, whose proof uses an argument that generalizes and slightly simplifies the argument from \cite{MM22} proving that a subspace of every near-vector space is itself a near-vector space.

The advantage of working in the more general context of Andr\'e modules is that it enables to establish the following answer to the question in the title: \emph{for the same reason as the category of modules over a ring is an abelian category}. Andr\'e modules are special types of modules over a (multiplicative) monoid $M$ (i.e., homomorphic actions of $M$ on an abelian group). Each class of Andr\'e modules is determined by a set $\mathbf{N}$ of near-ring structures on $M$, whose underlying multiplicative structure is given by $M$ (the pair $(M,\mathbf{N})$ is what we call a \emph{multi-near-ring} referred to above). If the set $\mathbf{N}$ is given by a single ring structure, then Andr\'e modules become usual modules over the corresponding ring. To get near-vector spaces, we must include all possible near-ring structures in the set $\mathbf{N}$ and furthermore, require that every element of $M$ that is not an additive identity in these near-ring structures is invertible (on the other end, this last restriction would have given us precisely vector spaces over a field). Now, the fact that the class of Andr\'e modules over a fixed multi-near-ring $(M,\mathbf{N})$ forms an abelian category rests on the fact that this class is a pseudovariety of $M$-modules (Theorem~\ref{abeliancategory}): it is closed under finite products, submodules and homomorphic images inside the class of all modules over the monoid $M$.

\section{Preliminaries}

Given a monoid $M$, an \emph{$M$-module} (or, a \emph{module over} $M$) is a pair $(V,\cdot)$ consisting of an (additive) abelian group $V$ and monoid action $\cdot$ by endomorphisms (to be referred as \emph{$M$-action}). This means that $\cdot$ is a map $M\times V\to V$ such
that the following identities hold: 
\begin{equation}
\label{1Eq1}
(\alpha\beta)\cdot  v=\alpha\cdot (\beta\cdot v),\quad 1\cdot v=v, \quad \alpha \cdot (u+v)= \alpha \cdot u +\alpha \cdot v.
\end{equation}
\begin{remark}
\label{notation}
When it is clear from the context, we call the pair $(V,\cdot)$ simply $V$. Note that throughout the paper, $V$ will be used to denote $V$ as an additive group as well as an $M$-module.
\end{remark}
As a consequence of the last identity in (\ref{1Eq1}), we have: 
\begin{equation}
\label{1Eq3}
    \alpha\cdot 0=0, \quad \alpha \cdot (-v) = -(\alpha \cdot v).
    \end{equation}
In an $M$-module $V$, for a subset $S\subseteq V$, we write $M\cdot S$ to denote
\[M\cdot S=\{\alpha\cdot s\mid \alpha\in M,\; s\in S \},\]
and we write $\overline{S}$ to denote the smallest subgroup of $V$ containing $S$. A \emph{submodule} of an $M$-module $V$ is a non-empty subset $S$ of $V$ such that $S$ is a subgroup of $V$, closed under the $M$-action. The smallest submodule of $V$ containing a given subset $S\subseteq V$ is given by $\overline{M\cdot S}$. In particular, $\overline{M \cdot \emptyset} = \{0\}$. 

The notion of a module over a monoid is certainly not new. In fact, any module over a monoid $M$ can be see as a module over the monoid ring $\mathbb{Z}[M]$ (see \cite{saunders}). 


Modules over a monoid form a variety of universal algebras, and hence this class of algebras has all the usual properties of a variety. In particular, it admits (infinite) products and sums, and more generally, all categorical limits and colimits. In this paper we are interested in subclasses of this class of algebras, which are not necessarily subvarieties.

\begin{definition}
Given a monoid $M$, an $M$-module $V$ is said to have the \emph{free action property} when the following holds for all $\alpha,\beta\in M$ and $v\in V$:
\begin{itemize}
    \item[(FA)] $[[\alpha\neq\beta]\land[\alpha\cdot v=\beta\cdot v]]\Rightarrow [v=0]$.
\end{itemize}
\end{definition}

\medskip
\begin{definition}
A \emph{scalar group} is a monoid $M$ having elements $0$ and $-1_{M}$ such that every element of $M\backslash \{0\}$ is invertible and the following two properties hold:
\begin{itemize}
   \item[(S0)] $0\cdot \alpha=0=\alpha\cdot 0$ holds for any $\alpha\in M$.
    \item[(S1)] $-1_{M}$ is such that $\eta=1$ and $\eta=-1_{M}$ are the only solutions of $\eta^2=1$.
    \end{itemize}
\end{definition}

Given any monoid $M$ with $0$ satisfying (S0), such an element $0$ is unique. Furthermore, given any monoid $M$ with $-1_{M}$ satisfying (S1), such an element $-1_{M}$ is unique. Indeed, if $\eta^2=1$ has only one solution, then $-1_{M}=1$. If $\eta^2=1$ has two solutions, then $-1_{M}$ must be the solution different from $1$.

A (left unitary abelian) \emph{near-ring} is an (additive) abelian group $R$ equipped with a (multiplicative) monoid structure, such that the multiplication makes $R$ into an $R$-module. We write a near-ring as a triple $R=(R,+,\cdot)$, where $+$ is the binary operation of the additive structure of $R$, and $\cdot$ is the binary operation of the multiplicative structure of $R$.  A \emph{near-field} is a near-ring $(F,+,\cdot)$ with $(F\backslash \{0\}, \cdot)$ being a group.

It is not difficult to see that in a near-ring with an underlying monoid having elements $0$ and $-1_{M}$ satisfying (S0) and (S1), the element $0$ must be the additive identity of the near-ring. However, in general, $-1_{M}$ will not be the additive inverse $-1$ of $1$, unless $-1$ differs from $1$. In the latter case, we must have $(-1)\cdot (-1)=-((-1)\cdot 1)=-(-1)=1$, and so $-1$ is the non-identity solution of $\eta^2=1$.  In a near-ring where $1=-1$, we have $a=a\cdot 1=a\cdot (-1)=-(a\cdot 1)=-a$. The following theorem provides an example where $-1 \neq -1_{M}$.

\begin{thm}
For a ring $R$ and a natural number $n$, consider the following binary operation on $R\times R$: 
$$(a,b)\#(c,d)=(ac,bc+a^nd).$$
The law $(ac)^n=a^nc^n$ holds in $R$ if and only if $(R\times R,+,\#)$ is a near-ring, where the addition $+$ is defined component-wise using the addition $+$ of the ring $R$. When $(R\times R,+,\#)$ is a near-ring, $(R\times R,+,\#)$ is a ring if and only if the law $(a+c)^n=a^n+c^n$ holds in $R$.
Finally, when $R$ is the two-element field $R=\mathbb{Z}_2$ and $n=1$, we get a ring whose underlying multiplicative monoid $M$ has elements $0$ and $-1_{M}$ satisfying (S0) and (S1), but $-1_{M}$ is not the additive inverse of the identity element of the ring.  
\end{thm}

\begin{proof}
Suppose $(ac)^{n} = a^{n}c^{n}$. First, we prove that the multiplication is associative.
\begin{align*}
(a,b)\#((c,d)\#(e,f)) &=(a,b)\#(ce,de+c^nf)\\
&=(ace,bce+a^nde+a^nc^nf)\\
&=(ac,bc+a^nd)\#(e,f)\\
&=((a,b)\#(c,d))\#(e,f)
\end{align*}
It is easy to see that the multiplicative identity is given by $(1,0)$. This proves that $(R \times R, \#)$ is a monoid. It remains to verify the right distributivity law:
\begin{align*}
(a,b)\#((c,d)+(e,f)) &=(a,b)\#(c+e,d+f)\\
&=(ac+ae,bc+be+a^nd+a^nf)\\
&=(ac+ae,bc+a^nd+be+a^nf)\\
&=(ac,bc+a^nd)+(ae,be+a^nf)\\
&=(a,b)\#(c,d)+(a,b)\#(e,f).
\end{align*}
Conversely, suppose $(R\times R, +, \#)$ is a near-ring.
Associativity of multiplication forces the law $(ac)^n=a^nc^n$:
\begin{align*}
(0,a^nc^n)&=(a,0)\#((c,0)\#(0,1))\\
&=((a,0)\#(c,0))\#(0,1)\\
&=(0,(ac)^n).
\end{align*}

Suppose $(R\times R,+,\#)$ is a near-ring. Then, under the assumption $(a+c)^n=a^n+c^n$, we have the left distributivity law:
\begin{align*}
((a,b)+(c,d))\#(e,f) &=(a+c,b+d)\#(e,f)\\
&=(ae+ce,be+de+a^nf+c^nf)\\
&=(ae+ce,be+a^nf+de+c^nf)\\
&=(ae,be+a^nf)+(ce,de+c^nf)\\
&=(a,b)\#(e,f)+(c,d)\#(e,f).
\end{align*}
Conversely, the left distributivity law forces $(a+c)^n=a^n+c^n$:
\begin{align*}
(0,(a+c)^n) &=((a,0)+(c,0))\#(0,1) \\
&=(a,0)\#(0,1)+(c,0)\#(0,1)\\
&=(0,a^n+c^n).
\end{align*}

Finally, consider the case when $R=\mathbb{Z}_2$ and $n=1$. Every element in $R\times R$ has additive order $2$. So, in particular, $1=(1,0)=-1$. However, we have $(1,1)\#(1,1)=(1,1+1)=(1,0)=1$ and $(0,1)\#(0,1)=(0,0)$. 
\end{proof}

\begin{definition}
An $M$-module $V$, where $M$ is a scalar group, is said to have the \emph{scalar action property} when the following holds for all $v\in V$:
\begin{itemize}
    \item[(SA)] $[0\cdot v=0]\land[(-1_M)\cdot v=-v]$.    
\end{itemize}
\end{definition}

A morphism of $M$-modules is defined as a morphism of underlying abelian groups that preserves the $M$-action. Morphisms can be composed in the usual way, giving rise to the category of $M$-modules, which we will denote by $M\textrm{-}\mathbf{Mod}$. This category is well known to be isomorphic to the categories of the following equivalent mathematical structures:
\begin{itemize}
    \item Functors from the monoid $M$ seen as a single-object category, to the category $\mathbf{Ab}$ of abelian groups. 
    
    \item Modules over the monoid ring $\mathbb{Z}[M]$.
\end{itemize}
This means, in particular, that each $M\textrm{-}\mathbf{Mod}$ is an abelian category.

 The properties (SA) and (FA) can be read as properties of a near-ring, where $0$ and $-1_M$ are interpreted, respectively, as the additive identity and the additive inverse of the multiplicative identity of the near-ring.

\begin{thm}
\label{LemA}
    Every near-field satisfies (FA). Every near-ring satisfying (FA) also satisfies (SA), where $-1$ is the additive inverse of $1$. Moreover, for every such near-ring, (S1) holds for $-1_{M} = -1$. Thus, in particular, the same is true for every near-field.
\end{thm}

\begin{proof} Suppose $a\cdot c=b\cdot c$ in a near-field, with $a\neq b$. If $c\neq 0$, then $a=a\cdot c\cdot c^{-1}=b\cdot c\cdot c^{-1}=b$. So $c=0$, proving that (FA) holds.

Consider now an arbitrary near-ring. In the case when $0=1$, all elements of the near-ring are equal to each other and hence both (FA) and (SA) trivially hold. Suppose $0\neq 1$. Then, since \[0\cdot (0\cdot \alpha)=(0\cdot 0)\cdot\alpha=0\cdot\alpha=1\cdot(0\cdot\alpha),\] we get that $0\cdot \alpha=0$ by (FA). If $-1=1$, then $-\alpha=\alpha\cdot(-1)=\alpha$, which completes the proof of (SA). If $-1\neq 1$, then because of
\[(-1)\cdot ((-1)\cdot \alpha+\alpha)=\alpha+(-1)\cdot \alpha=1\cdot((-1)\cdot \alpha+\alpha),\]
we get that $(-1)\cdot \alpha+\alpha=0$ and so $(-1)\cdot \alpha=-\alpha$, which completes the proof of (SA). Finally, suppose $\eta^2=1$. Then $\eta\cdot(\eta-1)=1-\eta=(-1)\cdot (\eta-1)$. If $\eta\neq 1$ then $\eta-1\neq 0$ and so, by (FA) we get $\eta=-1$. Conversely, both $\eta=1$ and $\eta=-1$ satisfy $\eta^2=1$ (that $\eta=-1$ satisfies $\eta^2=1$ follows easily from already established (SA)).
\end{proof}

Standard examples of near-rings that are not rings (and moreover, do not satisfy (SA)) are obtained as follows. Consider an abelian group $M$ with an underlying set $X$. The opposite of the monoid of endofunctions of $X$, i.e., the set of endofunctions with multiplication defined by $f\cdot g=g\circ f$, can be extended to a near-ring by point-wise addition defined according to the abelian group $M$,
$$(f+g)(x)=f(x)+_A g(x).$$
We will denote this near-ring by $\mathrm{Fun}(A)$. 
See \cite{Loc21} for a modern account of the theory of near-rings (and the references there for older texts on near-rings) as well as a discussion of examples of near-fields that are not fields. 

\section{Andr\'e modules}

We recall the original definition of a near-vector space \cite[Definition 4.1]{Andre}: 
\begin{definition}
\label{AndreNVS}
A pair $(V,F)$ is called a \emph{near-vector space} if:
\begin{enumerate}
    \item $V$ is an additive group and $F$ is a set of endomorphisms of $V$;
    \item $F$ contains the endomorphisms $0$, $1$ and $-1$, defined by $0 \cdot x = 0$, $1 \cdot x = x$ and $-1 \cdot x = -x$ for all $x \in V$;
    \item $F \setminus \{0\}$ is a subgroup of the group $\operatorname{Aut}(V)$;
    \item $F$ acts fixed point free on $V$, i.e. for $x \in V$, $\alpha, \beta \in F$, $\alpha x = \beta x$ implies that $x=0$ or $\alpha = \beta$;
    \item The quasi-kernel $Q(V)$ of $V$, generates $V$ as a group. Here, $$Q(V) = \{x \in V \mid \forall \alpha, \beta \in F, \exists \gamma \in F \text{ such that } \alpha x + \beta x = \gamma x \}.$$
\end{enumerate}
\end{definition}

If in the above definition, $F$ happens to be closed under component-wise addition of endomorphisms, then: 
\begin{itemize}
\item axioms (4) and (5) are consequences of (1-3), 
\item $F$ is a field,
\item and $V$ is a vector-space over $F$.
\end{itemize}
So, in some sense, we may think of a near-vector space as a vector space where the set of scalars is not closed under addition.

In this paper we will work with the following modified definition of a near-vector space, which fits better with the usual definition of a vector space as a module over a field, where the field of scalars is not forced to be a subset of the monoid of endofunctions of the set of vectors, and where, in addition, we allow $V$ to be a trivial group. This modification is basically a necessary language adjustment to allow us to define a morphism of near-vector spaces over the `same' $F$, and also, for the category of near-vector spaces to have a zero object. For a precise comparison of the two definitions, see Theorem~\ref{equivalent} and Remark~\ref{RemDefEquiv} below.

Although Definition~\ref{AndreNVS} excludes the case $V=\{0\}$, in \cite{Andre} the author systematically considers the possibility of $V=\{0\}$, which suggests that perhaps the author did not notice that $V\neq\{0\}$ is implied by axiom (3). The fact that the present formulation of (3) implies $V\neq\{0\}$ is noted in \cite{rabie}, where the issue is fixed by reformulating (3) as: if $F\setminus\{0\}$ is non-empty, then it is a subgroup of $\operatorname{Aut}(V)$.  

\begin{definition}
\label{newdefinition}
A near-vector space \emph{over} a scalar group $F$ is an $F$-module $V$ having the free and scalar action properties (i.e., satisfying (FA) and (SA)), such that $\overline{Q(V)}=V$, where 
\begin{equation}\label{EquA}
Q(V) = \{v \in V \mid \forall_{\alpha, \beta \in F} \exists_{\gamma \in F} [ \alpha \cdot v + \beta \cdot v = \gamma \cdot v ]\}.\end{equation}
\end{definition}

Note that $v\in Q(V)$ if and only if  $\overline{F\cdot\{v\}}=F\cdot\{v\}$, i.e., the $F$-orbit of $v$ is a subgroup of $V$.   

\begin{thm}
\label{equivalent}
    Consider a pair $(V,F)$ satisfying the conditions of Definition~\ref{AndreNVS}. Then $F$ is a scalar group under composition of endomorphisms of $V$, with $0,1$ and $-1$ being as in condition (2) of Definition~\ref{AndreNVS}. The action $F$ on $V$ given by evaluation, $\alpha x=\alpha(x)$, makes $V$ into a near-vector space over $F$ in the sense of Definition~\ref{newdefinition}. Conversely, given a near-vector space $V$ over a scalar group $F$, in the sense of Definition~\ref{newdefinition}, the pair $(V,F^\ast)$, where $V$ is the underlying abelian group of the given near-vector space and $F^\ast$ consists of those endofunctions of $V$ that are given by action of elements of $F$ on $V$, satisfies conditions of Definition~\ref{AndreNVS}, provided $V$ is not a trivial group.
\end{thm}

\begin{proof}
Suppose $(V,F)$ satisfies the conditions of Definition~\ref{AndreNVS}. Then $V$ is an additive group and $F$ is a submonoid of the endofunction monoid of $V$ (the monoid whose elements are endofunctions of $V$ and multiplication is composition of endofunctions) due to conditions (1) and (3). According to condition (3), every element of $F$ that is not the endofunction $0$ is required to be a bijection. Furthermore, by (3) again, identity function of $V$ belongs to $F$ and is distinct from $0$. Then, by Korollar to Satz 2.4 in \cite{Andre}, the monoid $F$ is the underlying multiplicative monoid of a near-field.
It then follows from Theorem~\ref{LemA} that the tuple $(F, \circ,1,0,-1)$ is a scalar group, where $\circ$ is the operation of composition of endofunctions that fall in $F$, while $1,0,-1$ are as in condition (2). When considering the action of $F$ on $V$ by evaluation, $\alpha x = \alpha (x)$, it must constitute an $F$-module since the endofunction $V\to V, 0:x\mapsto 0$ is clearly an additive group homomorphism, and the rest of the elements of $F$ are too thanks to condition (3). It is clear that conditions (2) and (4) in Definition~\ref{AndreNVS} are exactly what makes the action free and scalar, while condition (5) states exactly that $\overline{Q(V)} = V$.

Conversely, consider a near-vector space $(V,F)$ in the sense of Definition~\ref{newdefinition} and let $F^\ast$ be as in the statement of the theorem. Since $F$ is a scalar group, $F^\ast$ contains the endomorphisms $1:x\mapsto x, 0:x\mapsto 0$ and $-1:x\mapsto -x$, and so condition (2) of Definition~\ref{AndreNVS} holds for the pair $(V,F^*)$. Furthermore, since $V$ is not a trivial group, the endofunctions $0$ and $1$ differ, which implies that condition (3) holds. Finally, given that the action of elements of $F$ on $V$ is free, we get that condition (4) holds too. Conditions (1) and (5) hold trivially.
\end{proof}

\begin{remark}\label{RemDefEquiv}
Let $\mathbf{V}_1$ represent the collection of all pairs $(V,F)$ satisfying conditions of Definition~\ref{AndreNVS}, and let $\mathbf{V}_2$ be the collection of all pairs $(V',F')$ satisfying conditions of Definition~\ref{newdefinition}, with $V$ being a non-trivial group. The theorem above constructs two maps $I:\mathbf{V}_1\to \mathbf{V}_2$ and $J:\mathbf{V}_2\to \mathbf{V}_1$. It is clear that the composite $JI$ is the identity map $\mathbf{V}_1 \to \mathbf{V}_1$. For the other composite, we have that $IJ(V',F')=(V',F'')$, where, thanks to the the free action property, $F''$ is isomorphic to $F'$, and moreover, this isomorphism is compatible with the actions of $F'$ and $F''$ on $V'$. This clarifies the sense in which the two definitions are equivalent.  
\end{remark}



What follows in this paper is motivated by an attempt to identify, for a monoid $M$, abelian subcategories of $M\textrm{-}\textbf{Mod}$ using simple axioms on an $M$-module in such a way as to fulfill the following two requirements:
\begin{itemize}
    \item When $M$ is a scalar group, these axioms determine precisely the category of near-vector spaces.
    \item When $M$ is the underlying multiplicative monoid of a ring $R$, these axioms identify precisely the modules over the ring $R$.
\end{itemize}
This is, however, not possible, since already for a field $F$, the first requirement above says that such axioms would determine the category of near-vector spaces over $F$, whereas the second requirement above says that such axioms would determine the category of vector spaces over $F$. We can, however, remedy this by equipping $M$ with an extra structure.

\begin{definition}
A \emph{multi-near-ring} is a multiplicative monoid $M$ equipped with a set of designated near-rings  whose underlying multiplicative monoids are all given by the monoid $M$.
\end{definition}

Thus, more formally, a multi-near ring $R$ is a pair $R=(M_R,\mathbf{N}_R)$, consisting of a monoid $M_R$ and a set $\mathbf{N}_R$ such that each element $
N$ of $\mathbf{N}_R$ is a near-ring with $M_{R}$ as the underlying multiplicative monoid (and we call the elements of $\mathbf{N}_R$ the \emph{designated} near-rings of the multi-near-ring $R$). We write $+_N$ for the operation of addition in each near-ring $N$. 
 


For a multi-near-ring $R=(M_R,\mathbf{N}_R)$, we will write the same $R$ to denote the underlying set of the monoid $M_R$, which is the same as the underlying set of each $S\in\mathbf{N}_R$.



The notion of a module over a multi-near-ring $R=(M_R,\mathbf{N}_R)$ that we propose below will be such that: 
\begin{description}
\item[Claim 1] For any multi-near-ring $R$, the category of modules over $R$ is an abelian category (see Theorem \ref{abeliancategory}).

\item[Claim 2] In the case when $\mathbf{N}_R$ is a set consisting of a single ring $N$, the modules over $R$ will be exactly the same as modules over the ring $N$ (see Theorem \ref{singleR}).

\item[Claim 3] In the case when $M_R$ is a scalar group and $\mathbf{N}_R$ is the set of all near-rings with $M_R$ as the underlying multiplicative monoid, the modules over $R$ will be precisely the near-vector spaces over the scalar group $M_R$ (see Theorem \ref{Andremodule}).
\end{description}
\begin{definition}
An \emph{(Andr\'e) module} $V$ over a multi-near-ring $R=(M_R,\mathbf{N}_R)$ is an $M_R$-module such that there exists a subset $Q\subseteq V$ having the following two properties:
\begin{itemize}
    \item[(QK1)] For every element $0\neq v\in Q$ there is a near-ring $N_v
    \in
    \mathbf{N}_R$ such that $(\alpha+_v \beta)\cdot v=\alpha\cdot v+\beta\cdot v$ holds for all $\alpha,\beta\in R$, where `$+_v$' stands for `$+_{N_v}$'.
    
    \item[(QK2)] For any $v\in V$, we have: $$v\in \overline{M_R\cdot(\overline{M_R\cdot \{v\}}\cap Q)}.$$
\end{itemize}
\end{definition}

It is not difficult to get convinced that (QK1) and (QK2) do not imply each other. Indeed, if $Q=\{0\}$ then (QK1) always holds, but (QK2) forces $V=\{0\}$. 
If however $Q=V$, then (QK2) always holds, but not necessarily (QK1) (e.g.,~when $\mathbf{N}_R=\varnothing$, (QK1) forces $V=\{0\}$).



That the subcategory of $M_R\textrm{-}\mathbf{Mod}$ consisting of Andr\'e modules over a multi-near-ring $R$ is an abelian category follows from the fact that $M_R\textrm{-}\mathbf{Mod}$ is an abelian category and that this subcategory is closed under submodules, homomorphic images of modules and finite products (which we prove below), thus fulfilling Claim~1.
\begin{thm}
\label{abeliancategory}
The class of Andr\'e modules over a multi-near-ring $R$ is closed under finite products, submodules, and homomorphic images. Hence, the full subcategory of $M_R\textrm{-}\mathbf{Mod}$ consisting of Andr\'e modules is an abelian category. 
\end{thm}

\begin{proof}
 The empty product of Andr\'e modules is the Andr\'e module where $V$ is a singleton.
In this case the properties (QK1) and (QK2) hold trivially for $Q=\varnothing$   (as well for $Q=V$). Given two Andr\'e modules, $V_1$ and $V_2$, each with $Q_1$ and $Q_2$ satisfying (QK1) and (QK2), set $Q=(Q_1\times\{0\})\cup (\{0\}\times Q_2)$. It is easy to see that (QK1) holds. For (QK2), first decompose $v=(v_1,v_2)$ as $v=(v_1,0)+(0,v_2)$ and then apply (QK2) for $Q_1$ and $Q_2$ in the respective coordinates. This proves that the class of Andr\'e modules is closed under finite products.

Consider a submodule $W$ of an Andr\'e module $V$. If $Q$ is the subset of $V$ that satisfies (QK1) and (QK2), then $Q\cap W$ will serve as such subset of $W$.  Indeed, that (QK1) holds for $Q\cap W$ is obvious. For (QK2), we note that when $v\in W$, we have the following (note that we write $R$ below for $M_R$) $$v\in \overline{R\cdot(\overline{R\cdot \{v\}}\cap Q)}=\overline{R\cdot(\overline{R\cdot \{v\}}\cap (Q\cap W))}.$$
This argument is valid because additive closure and orbits are given in $W$ in the same way as in $V$, since $W$ is a submodule of $V$.

Consider now a surjective morphism of modules $f\colon V\to V'$, where $V$ is an Andr\'e module. Let $Q\subseteq V$ be such that (QK1) and (QK2) hold. We will show that $f(Q)=\{f(q)\mid q\in Q\}$ will make (QK1) and (QK2) hold for $V'$. In fact, (QK1) follows trivially from the same property for $Q$ and the fact that $f$ is a module morphism. (QK2) is also each to check:
\begin{align*}
    f(v) &\in f(\overline{R\cdot(\overline{R\cdot \{v\}}\cap Q)})\\
    &= \overline{R\cdot(f(\overline{R\cdot \{v\}}\cap Q))}\\
    &\subseteq \overline{R\cdot(f(\overline{R\cdot \{v\}})\cap f(Q))}\\
    &=\overline{R\cdot(\overline{R\cdot \{f(v)\}}\cap f(Q))}.\qedhere
\end{align*}
\end{proof}
Since the category of Andr\'{e} modules is stable under binary products, one can take a finite product of near-rings (seen as Andr\'{e} modules over the same multi-near-ring) to create a new Andr\'{e} module.
\begin{thm}
\label{+N}
    Let $R = (M_{R},\mathbf{N}_{R})$ be a multi-near-ring. Then, the product $\Pi_{N \in \mathbf{L}}N$ of $M_{R}$-modules, where $\mathbf{L}$ is a finite subset of $\mathbf{N}_R$, is an Andr\'{e} module over $R$.  
\end{thm}
\begin{proof}
    Let $N \in \mathbf{L}$. Then, one can see that $N$ will be an Andr\'e module over $R$ by setting $Q=\{1\}$ in (QK1-2). By Theorem \ref{abeliancategory}, taking the product $\Pi_{N \in \mathbf{L}}N$ is an Andr\'{e} module over $R$ with $Q = \{e_{N} \mid N \in \mathbf{L}\} $ where $e_{N} = (\delta_{MN})_{N \in \mathbf{L}}$ and
    $$\delta_{MN}=\left\{\begin{array}{cc} 1, & M=N,\\ 0, & M\neq N. \end{array}\right.$$
    is the Kronecker delta function.
\end{proof}

As a consequence of Theorem \ref{+N}, as soon as we have two near-rings (which are not near-fields) having the same scalar multiplication but with distinct addition, we can create an example of an Andr\'{e} module which is not a near-vector space by taking the product of two such near-rings. For example, define the map $\phi: \mathbb{Z}/9\mathbb{Z} \rightarrow \mathbb{Z}/9\mathbb{Z}$ by sending $[3]^{u}s$ to $[6]^{u}s$ where $u \in \{0,1,2\}$ and $s$ is coprime to $[3]$. Then, $\phi$ is a multiplicative automorphism. Now, we can define an addition induced by $\phi$ on $\mathbb{Z}/9\mathbb{Z}$ by defining $[m]+_{\phi}[n] = \phi^{-1}(\phi([m])+\phi([n]))$ with $[m],[n] \in \mathbb{Z}/9\mathbb{Z}$. Then, $(\mathbb{Z}/9\mathbb{Z},+_{\phi},\cdot)$ is a near-ring. Let $V=(\mathbb{Z}/9\mathbb{Z},+,\cdot)\times (\mathbb{Z}/9\mathbb{Z},+_{\phi},\cdot)$ and $R = ((\mathbb{Z}/9\mathbb{Z},\cdot), \{(\mathbb{Z}/9\mathbb{Z},+,\cdot),(\mathbb{Z}/9\mathbb{Z},+_{\phi},\cdot) \})$ with $V$ having scalar multiplication component-wise and  addition
\begin{align*}
    ([m],[n])\oplus ([m'],[n']) = ([m]+[m'],[n]+_{\phi}[n'])
\end{align*}
where $[m],[m'],[n],[n']\in \mathbb{Z}/9\mathbb{Z}$. Then, $V$ is an Andr\'{e} module over $R$. However, $(\mathbb{Z}/9\mathbb{Z},+,\cdot)$ is not a near-field, since $\mathbb{Z}/9\mathbb{Z}$ contains zero divisors. 
For this reason, $V$ is not a near-vector space over $R$.

Claim 2 is fulfilled with the following theorem:

\begin{thm}
\label{singleR}
An Andr\'e module over a ring $R$ (seen as the multi-near-ring $(M_R,\{R\})$, where $M_R$ is the underlying multiplicative monoid of the ring $R$) is the same as a module over $R$ in the usual sense. 
\end{thm}

\begin{proof}
Consider a ring $R$ seen as the multi-near-ring $R=(M_R,\{R\})$. If $V$ is a module over the ring $R$, we can see that $V$ will be an Andr\'e module over $R$ by setting $Q=V$. 
Conversely, consider an Andr\'e module $V$ over $R$. Let $Q\subseteq V$ be such that (QK1) and (QK2) hold. To show that $V$ is a module over the ring $R$, all we need to show is that $(\alpha+\beta)\cdot v=\alpha\cdot v+\beta\cdot v$ holds in $V$ for all $v\in V$ and for all $\alpha, \beta \in R$. We know that this must be the case for all $0\neq v\in Q$ thanks to (QK1). Consider an arbitrary $v\in V$. The case $v=0$ is trivial. Suppose $v\neq 0$. By (QK2), $v=\gamma_1\cdot v_1+\dots+\gamma_n\cdot v_n$ for some $\gamma_1,\dots,\gamma_n\in R$ and $v_1,\dots,v_n\in Q$, where $v_i\neq 0$ for each $i\in\{1,\dots,n\}$. Then:
\begin{align*}
(\alpha+\beta)\cdot v &= (\alpha+\beta)\cdot (\gamma_1\cdot v_1+\dots+\gamma_n\cdot v_n) \\
&=(\alpha+\beta)\cdot (\gamma_1\cdot v_1)+\dots+(\alpha+\beta)\cdot(\gamma_n\cdot v_n)\\
&=((\alpha+\beta)\cdot \gamma_1)\cdot v_1+\dots+((\alpha+\beta)\cdot\gamma_n)\cdot v_n\\
&=(\alpha\cdot \gamma_1+\beta\cdot \gamma_1)\cdot v_1+\dots+(\alpha\cdot\gamma_n+\beta\cdot\gamma_n)\cdot v_n\\
&=\alpha\cdot \gamma_1\cdot v_1+\beta\cdot \gamma_1\cdot v_1+\dots+\alpha\cdot\gamma_n\cdot v_n+\beta\cdot\gamma_n\cdot v_n\\
&=\alpha\cdot \gamma_1\cdot v_1+\dots+\alpha\cdot \gamma_n\cdot v_n +\beta\cdot \gamma_1\cdot v_1+\dots+\beta\cdot\gamma_n\cdot v_n\\
&=\alpha\cdot (\gamma_1\cdot v_1+\dots+\gamma_n\cdot v_n) +\beta\cdot (\gamma_1\cdot v_1+\dots+\gamma_n\cdot v_n)\\
&=\alpha\cdot v+\beta\cdot v.\qedhere
\end{align*}
\end{proof}

\begin{remark}
    The theorem above remains true for a near-ring in place of a ring, provided by a module over a near-ring we mean a module in the sense of module over a ring.
\end{remark}

We now fulfill Claim~3. This will make use of what we shall refer to as the \emph{subspace hypothesis for near-vector spaces}, which states that every submodule of a near-vector space is itself a near-vector space under the induced structure. This property reduces to requiring that $W=\overline{Q(W)}$ for any submodule $W$. 

The subspace hypothesis for near-vector spaces has been proven in \cite{MM22} (and in a special case in \cite{HM22}) using a purely algebraic argument, and in \cite{rabie,HR22} using a geometric approach to near-vector spaces. In the next section, we provide a simplified and generalized version of the algebraic argument from \cite{MM22}.


\begin{thm}
\label{Andremodule}
For a scalar group $F$, a near-vector space over $F$ is the same as an Andr\'e module over $(F,\mathbf{N})$, where $\mathbf{N}$ is the set of all possible near-rings with $F$ as the underlying multiplicative monoid.
\end{thm}

\begin{proof}
Let $F$ be a scalar group and let $V$ be an Andr\'e module over $(F,\mathbf{N})$, where $\mathbf{N}$ is the same as in the statement of the theorem. Note that every element of $\mathbf{N}$ is a near-field. Indeed, since $(F\backslash \{0\},\cdot)$ is a group and $F$ is the underlying multiplicative monoid of every near-ring in $\mathbf{N}$, for all $N \in \mathbf{N}$, every non-zero element of $N$ has a multiplicative inverse, therefore $N$ is a near-field. We will prove that $V$ is a near-vector space over $F$. Let $Q\subseteq V$ be such that (QK1) and (QK2) hold and let $Q(V)$ be defined as in (\ref{EquA}). Then clearly $Q\subseteq Q(V)$ by (QK1). It is not hard to see that $F\cdot Q(V)=Q(V)$, and so $F\cdot Q\subseteq Q(V)$. This, together with (QK2), gives that $\overline{Q(V)}=V$. Note that in fact, by (QK2), $\overline{F\cdot Q}=V$. To prove the property (FA), suppose $\alpha\cdot v=\beta\cdot v$ where $v\neq 0$. Write $v$ as $v=\gamma_1\cdot v_1+\cdots+\gamma_n\cdot v_n$, where $v_1,\dots,v_n\in Q$ and $\gamma_1,\cdots,\gamma_n \in F$. Moreover, choose $v_1,\dots,v_n$ in such a way that $n$ is smallest with the property that such presentation of $v$ is possible. Then $\gamma_i\neq 0$ and $v_i\neq 0$ for all $i\in\{1,\dots,n\}$, as well as $n\neq 0$. From $\alpha\cdot v=\beta\cdot v$ we have:
$$(\alpha\cdot \gamma_1-_{v_1}\beta\cdot \gamma_1)\cdot v_1+\dots+(\alpha\cdot \gamma_n-_{v_n}\beta\cdot \gamma_n)\cdot v_n=0.$$
If $\alpha\cdot \gamma_1-_{v_1}\beta\cdot \gamma_1=0$ then, since $\gamma_i\neq 0$, we will get that $\alpha=\beta$, as desired. If $\alpha\cdot \gamma_1-_{v_1}\beta\cdot \gamma_1\neq 0$, then we can express $v_1$ as a linear combination of $v_2,\dots,v_n$ which would enable a presentation of $v$ via $v_2,\dots,v_n$, one less term than $n$, which is not possible. This completes the proof of (FA).

Next, we prove (SA). Consider $v\in V$. If $v=0$, then trivially, $0\cdot v=0$ and $(-1)\cdot v=-v$. If $v\neq 0$, then we can decompose $v$ as a linear combination of elements of $Q$ and apply the fact that (SA) holds for every element of $Q$.

So $V$ is a near-vector space over $F$. Now let us establish the converse. Suppose $V$ is a near-vector space over $F$. To see that $V$ is an Andr\'e module over $(F,\mathbf{N})$, we will show that (QK1) and (QK2) hold for $Q=Q(V)$. In fact, (QK1) is known already from \cite{Andre}. For (QK2), consider the submodule $W=\overline{F\cdot \{v\}}$ of $V$. It is easy to see that $Q(W)=W\cap Q(V)$. So, by the subspace hypothesis for near-vector spaces,
\[v\in W=\overline{Q(W)}=\overline{W\cap Q(V)}=\overline{\overline{F\cdot \{v\}}\cap Q(V)}.\]
Since $\overline{\overline{F\cdot \{v\}}\cap Q(V)} \subseteq \overline{F\cdot (\overline{F\cdot \{v\}}\cap Q(V))}$, it follows that $v \in \overline{F\cdot (\overline{F\cdot \{v\}}\cap Q(V))}$.
\end{proof}

\section{A direct proof of the subspace hypothesis for near-vector spaces}

Firstly, let us remark that the subspace hypothesis for near-vector spaces is equivalent to (QK2) for $Q=Q(V)$, as this should already be apparent from the previous section. In fact, since $R\cdot Q(V)=Q(V)$ (where $R$ denotes the given scalar group), we have the following theorem. 
\begin{thm}
\label{ThmAA}
    Let $V$ be a near-vector space over a scalar group $R$. Then:
$$\overline{\overline{R\cdot \{v\}}\cap Q(V)} = \overline{R\cdot (\overline{R\cdot \{v\}}\cap Q(V))}.$$
\end{thm}
\begin{proof}
    It is clear that $\overline{\overline{R\cdot \{v\}}\cap Q(V)} \subseteq \overline{R\cdot (\overline{R\cdot \{v\}}\cap Q(V))}$. Let $v \in \overline{R\cdot (\overline{R\cdot \{v\}}\cap Q(V))}$. We can write $v$ as $v = \gamma_{1}\cdot v_{1}+\cdots +\gamma_{n}\cdot v_{n}$, where $v_{1},\cdots, v_{n} \in \overline{R\cdot \{v\}}\cap Q(V)$ and $\gamma_{1},\cdots, \gamma_{n} \in R$. Then, $\gamma_{i} \cdot v_{i} \in \overline{R \cdot \{v\}}\cap Q(V)$ for each $i \in \{1,\cdots,n\}$ (since $\overline{R \cdot \{v\}}\cap Q(V)$ is closed under scalar multiplication). Therefore it follows that $v \in \overline{\overline{R\cdot \{v\}}\cap Q(V)}$ and so $$\overline{\overline{R\cdot \{v\}}\cap Q(V)} = \overline{R\cdot (\overline{R\cdot \{v\}}\cap Q(V))}, $$ 
    which concludes the proof.
\end{proof}
Thanks to Theorem \ref{ThmAA}, we can prove the following theorem. 
\begin{thm}
\label{TFAE}
Let $V$ be a near-vector space over a scalar group $R$. Then the following are equivalent:
\begin{enumerate}
    \item Every submodule of $V$ is itself a near-vector space under the induced structure.
    \item For any $v \in V$, we have:
    $$v \in \overline{R\cdot( \overline{R\cdot \{v\}}\cap Q(V)}.$$
    \item For every $v \in V$, we have:
    $$v \in \overline{\overline{R\cdot \{v\}}\cap Q(V)}.$$
\end{enumerate}
\end{thm}
\begin{proof}
$(1) \Rightarrow (2)$: By Theorem \ref{Andremodule}, we know that $V$ is the same as an Andr\'{e} module $(F, \mathbf{N})$, where $\mathbf{N}$ is the set of all possible near-rings with $F$ as the underlying multiplicative monoid. Therefore, by (QK2), for any $v\in V$ we have
$$v\in  \overline{R\cdot( \overline{R\cdot \{v\}}\cap Q}$$
where $Q$ is from the definition of an Andr\'{e} module. However, $Q\subseteq Q(V)$ follows easily from (QK1), which then gives:
$$v\in  \overline{R\cdot( \overline{R\cdot \{v\}}\cap Q(V)}.$$

$(2)\Rightarrow (3)$ by Theorem \ref{ThmAA}. 

$(3) \Rightarrow (1)$: Let $W$ be a submodule of $V$. $W$ inherits the free and scalar action properties from $V$. It is clear that $\overline{Q(W)} \subseteq W$ (this is because $W$ is closed under addition). It is also clear that $Q(W) = W\cap Q(V)$. Let $w \in W$. By $(3)$, $w \in \overline{\overline{R\cdot \{w\}}\cap Q(V)}$. Since $W$ is closed under scalar multiplication, $R\cdot \{w\} \subseteq W$ and so $\overline{R\cdot \{w\}} \subseteq W$. Therefore, $\overline{R\cdot \{w\}}\cap Q(V) \subseteq W \cap Q(V)$ so that $$\overline{\overline{R\cdot \{w\}}\cap Q(V)} \subseteq \overline{W \cap Q(V)} = \overline{Q(W)}$$ and so $W \subseteq \overline{Q(W)}$. 
\end{proof}
We could rewrite (QK2) stated for $Q=Q(V)$ as follows:
\begin{itemize}
    \item[(QK2$'$)] For every $v\in V$, we have: \begin{equation}\label{EquC}
    v\in \overline{\overline{R\cdot \{v\}} \cap Q(V)}.\end{equation}
\end{itemize}
We will deduce (QK2) from (QK1) (which we already know holds for a near-vector space) and the following property, whose validity in the case of a scalar group $R$ is trivial: 
\begin{itemize}
    \item[(QK3)] For any $v\in V\setminus Q$, $q\in Q$, $0\neq \alpha\in R$ and $r\in V$, if $\alpha\cdot q+r\in \overline{R\cdot\{v\}}$ then there exists $\beta\in R$ such that $q+\beta\cdot r\in \overline{R\cdot\{v\}}$.
\end{itemize}
 We will actually do this more generally when $R$ is not necessarily a scalar group.

\begin{thm}\label{ThmA}
Let $R=(M_R,\mathbf{N}_R)$ be a multi-near-ring. Consider the subset $Q(V)$ (defined by (\ref{EquA}) with $F=M_R$) of an $M_R$-module $V$. If 
\begin{itemize}
    \item $Q=Q(V)$ satisfies (QK1),
    
    \item $M_R\cdot Q(V)=Q(V)$,
    
    \item and $\overline{Q(V)}=V$,
\end{itemize}
then $Q=Q(V)$ will satisfy (QK2$'$) (and hence (QK2)) as soon as (QK3) holds.
\end{thm}

\begin{proof} The argument is by strong induction on the minimal number $m_v$ for which $v$ can be presented as a sum $v=q_1+\dots+q_{m_v}$, where $q_1,\dots,q_{m_{v}}\in Q(V)$. Let us assume that (\ref{EquC}) holds for all $v$ such that $m_v<n$. Suppose $m_v=n$. If $n=0$ or $n=1$, then (\ref{EquC}) holds trivially. Suppose $n>1$. Then $q_1\neq q_{2}$, since otherwise $n\neq m_v$. Moreover, for the same reason, $+_{q_1}\neq +_{q_{2}}$, since if $+_{q_1}=+_{q_{2}}$ then it is easy to show that $q_1+q_{2}\in Q(V)$. Indeed, the following well-known argument from the theory of near-vector spaces applies: if $+_{q_1}=+_{q_{2}}$ then for any $\alpha,\beta\in R$,
\[\alpha\cdot (q_1+q_2)+\beta\cdot (q_1+q_2)= (\alpha+_{q_1}\beta)\cdot q_1+(\alpha+_{q_2}\beta)\cdot q_2=(\alpha+_{q_1}\beta)\cdot (q_1+q_2),
\]
proving that $q_1+q_2$ fulfills the requirement for inclusion in $Q(V)$. So suppose $+_{q_1}\neq +_{q_{2}}$. Then there exists $\alpha,\beta\in R$ such that $\alpha+_{q_1}\beta\neq \alpha+_{q_2}\beta$. Note that 
\[\alpha\cdot v+\beta\cdot v-(\alpha+_{q_1}\beta)\cdot v=\sum_{i=2}^{n}((\alpha+_{q_i}\beta)-_{q_i}(\alpha+_{q_1}\beta))\cdot q_i.\]
Since $m_v=n>1$, we must have $v\in V\setminus Q(V)$. We can then apply (QK3) in the case when $q=q_2$, $\alpha$ of (QK3) is the coefficient of $q_2$ in the above sum and \[r=\sum_{i=3}^{n}((\alpha+_{q_i}\beta)-_{q_i}(\alpha+_{q_1}\beta))\cdot q_i.\]
By (QK3) we then get that there exists $\beta'\in R$ such that 
\begin{equation}
\label{EquD}
q_2+\beta'\cdot r=q_2+\sum_{i=3}^{n}\beta'\cdot((\alpha+_{q_i}\beta)-_{q_i}(\alpha+_{q_1}\beta))\cdot q_i\in \overline{R\cdot\{v\}}.
\end{equation}
We then have: 
\begin{align*}
v&=q_1+\left(q_2+\sum_{i=3}^{n}\beta'\cdot((\alpha+_{q_i}\beta)-_{q_i}(\alpha+_{q_1}\beta))\cdot q_i\right)-\sum_{i=3}^{n}\beta'\cdot((\alpha+_{q_i}\beta)-_{q_i}(\alpha+_{q_1}\beta))\cdot q_i+\sum_{i=3}^{n}q_i\\
&=q_1+(q_2+\beta'\cdot r)+\sum_{i=3}^{n} (1-_{q_i}\beta'\cdot((\alpha+_{q_i}\beta)-_{q_i}(\alpha+_{q_1}\beta)))\cdot q_i.
\end{align*}
This means that \begin{equation}\label{EquE}v'=v-(q_2+\beta'\cdot r)\end{equation} can be expressed as a sum of $n-1$ terms from $Q(V)$. By the induction assumption, we have the following (where we write $R$ for $M_R$): \[v'\in \overline{\overline{R\cdot \{v'\}} \cap Q(V)}.\]
Note however that $v'\in\overline{R\cdot\{v\}}$, by (\ref{EquD}) and (\ref{EquE}). So \[ \overline{\overline{R\cdot \{v'\}} \cap Q(V)}\subseteq \overline{\overline{R\cdot \{v\}} \cap Q(V)},\]
which gives us that 
\[v'\in \overline{\overline{R\cdot \{v\}} \cap Q(V)}.\]
Since $v''=q_2+\beta'\cdot r$ is a sum of $n-1$ elements of $Q(V)$, by the induction assumption, 
\[v''\in \overline{\overline{R\cdot \{v''\}}\cap Q(V)}.\]
Similarly as for $v'$, we then have
\[v''\in \overline{\overline{R\cdot \{v\}} \cap Q(V)}.\]
Since $v=v'+v''$, we get (\ref{EquC}), as desired.  
\end{proof}

\end{document}